\documentclass[reqno]{amsart}
\usepackage{amssymb,amscd,verbatim, amsthm,graphics, color,latexsym,amsmath,multicol}
\usepackage{fancybox}
\usepackage{longtable}

\usepackage[all]{xy}

\newcommand \fk[1]{{{\mathfrak #1}}}
\newcommand \C[1]{{\mathcal #1}}

\newcommand \wti[1]{{\widetilde {#1}}}

\newcommand \bC{{\mathbb C}}

\newcommand \bH{{\mathbb H}}

\newcommand \bZ{{\mathbb Z}}

\newcommand\ep{{\epsilon}}

\newcommand\om{{\omega}}
\newcommand\al{{\alpha}}

\newcommand\fh{{\mathfrak h}}

\newtheorem{theorem}{Theorem}[section]

\newtheorem{corollary}[theorem]{Corollary}
\newtheorem{lemma}[theorem]{Lemma}
\newtheorem{proposition}[theorem]{Proposition}
\newtheorem{definition}[theorem]{Definition}
\newtheorem{remark}[theorem]{Remark}

\newcommand\Hom{\operatorname{Hom}}

\newcommand\Ind{\operatorname{Ind}}

\newcommand\im{\operatorname{im}}

\newcommand\triv{\mathsf{triv}}
\newcommand\sgn{\mathsf{sgn}}
\newcommand\refl{\mathsf{refl}}

\newcommand\Irr{\mathsf{Irr}}

\newcommand\cusp{\mathsf{cusp}}

\newcommand\bfH{\mathbf {H}}

\newcommand\Id{\operatorname{Id}}
\newcommand\Spec{\operatorname{Spec}}

\def\<{\langle} 
\def\>{\rangle}

\numberwithin{equation}{subsection}

\begin{document}

\title[One-$W$-type modules and cuspidal two-sided cells]{One-$W$-type modules for rational Cherednik algebra and cuspidal two-sided cells}

\author{Dan Ciubotaru}
        \address[D. Ciubotaru]{Mathematical Institute\\ University of
          Oxford\\ Oxford, OX2 6GG, UK}
        \email{dan.ciubotaru@maths.ox.ac.uk}

\begin{abstract}
We classify the simple modules for the rational Cherednik algebra $\bfH_{0,c}$ that are irreducible when restricted to $W$, in the case when $W$ is a finite Weyl group. The classification turns out to be closely related to the cuspidal two-sided cells in the sense of Lusztig. We compute the Dirac cohomology of these modules and use the tools of Dirac theory to find nontrivial relations between the cuspidal Calogero-Moser cells and the cuspidal two-sided cells. 
\end{abstract}

\maketitle

\setcounter{tocdepth}{1}
\tableofcontents

\section{Introduction}
In this paper, we classify the simple modules for the rational Cherednik algebra $\bfH_{0,c}$, $c$ arbitrary, that are irreducible when restricted to $W$, in the case when $W$ is a finite Weyl group. We find that such modules exist when $W$ is of type $B_n$, $D_n$, $G_2$, $F_4$, $E_6$, and $E_8$, but not in type $A_{n-1}$ (as expected from \cite{EG}) or in type $E_7$. The classification result is Theorem \ref{t:class}. 

In \cite{Ci}, we introduced a Dirac operator and the notion of Dirac cohomology in the setting of the graded Hecke algebras defined by Drinfeld, extending in this way the construction from \cite{BCT} for Lusztig's graded affine Hecke algebras. In particular, the constructions in \cite{Ci} apply to rational Cherednik algebras. The Dirac cohomology of one-$W$-types turns out to be easy to compute and yet it yields a good amount of nontrivial information. The results of this paper can be viewed as the rational Cherednik algebra analogues of the results about one-$W$-type modules for the graded affine Hecke algebra from \cite{BM} and \cite{CM}.

The main application of this method is that it gives a direct relation between the cuspidal components in the partition of $\Irr W$ coming from considering certain fibers over the Calogero-Moser (CM) space $X_c(W)$ of $\bfH_{0,c}$ (see for example \cite{EG}, \cite{GM}, \cite{Go}), and the cuspidal two-sided cells in $\Irr W$, in the sense of Lusztig \cite{Lu}. 
The CM partition of $\Irr(W)$ is expected (see \cite{GM}) to be related to the one into two-sided cells (or families) of representations, in the sense of Lusztig \cite{Lu} for real reflection groups, and \cite{Ro} for complex reflection groups. At least when $W$ is a real reflection group, the conjecture is that the two partitions coincide. This is  known to hold in several cases, for example in type $A$ \cite{EG}, and for the class of complex reflection groups $G(m,d,n)$ by \cite{GM}, \cite{Be}, and \cite{Ma}. See also \cite{BR} for more details as well as for a refinement of this conjecture.

\smallskip

Assume now that the rational Cherednik algebra has equal parameters. In Theorem \ref{t:CM-cusp}, using the approach via the Dirac operator and one-$W$-types, we find nontrivial relations between the cuspidal CM cells and Lusztig's cuspidal families. More precisely, recall that cuspidal families exist only for $B_n$, $n=d^2+d$, $D_n$, $n=d^2$, $G_2$, $F_4$, $E_6$, $E_7$, and $E_8$, and in each of these cases there exists only one such family, which we denote by $\C F_\cusp(W)$.  For example, when $W=E_8$, there are $17$ representations in $\C F_\cusp(W).$

Let 
$$\Theta:\Irr(W)\to X_1(W)$$
be the morphism defined by \cite{Go}, more details are in the body of the paper. It turns out, see Theorem \ref{t:class}, that $\bfH_{0,1}$ admits one-$W$-type modules if and only if $W$ has a cuspidal two-sided cell {\it except} when $W=E_7$, in which case $\bfH_{0,1}$ does not have one-$W$-type modules. When one-$W$-type modules for $\bfH_{0,1}$ exist, they correspond to the same point in $X_1(W)$ which we denote $\mathbf 0$.

\begin{theorem}[{also Theorem \ref{t:CM-cusp}}]\label{t:intro}
Let $W$ be a simple finite Weyl group and let $\bfH_{0,1}$ be the rational Cherednik algebra with equal parameters. Suppose $W$ has a cuspidal two-sided cell.
\begin{enumerate}
\item If $W$ is $B_n$, $D_n$, $G_2$, $F_4$, or $E_6$, then  
\[\C F_\cusp(W)\subseteq \Theta^{-1}(\mathbf 0).\]
(When $W=G_2$, $F_4$, or $E_6$, this is an equality.)
\item If $W=E_8$, then 
\[\C F_\cusp(E_8)\setminus\{4480_y\}\subseteq \Theta^{-1}(\mathbf 0)\subseteq \C F_\cusp(E_8)\cup\{2100_y\}.\]
\item If $W=E_7$, then
\[\{512_a'\}\subseteq \Theta^{-1}(\Theta(512_a'))\subseteq \{512_a',512_a\}=\C F_\cusp(E_7).\]
\end{enumerate}
\end{theorem}

The notation for $W$-representations is as in \cite{Lu,Ca}. The claim about $E_7$ is straight-forward but we included it for convenience. The proof of Theorem \ref{t:intro} uses only Theorem \ref{t:class} and elements of Dirac cohomology. While the method of the proof is uniform, certain details, such as decomposition of tensor products, have to be checked case by case. The method also applies to unequal parameter cases, and likely to complex reflection groups as well, but we do not consider those analogues of Theorem \ref{t:intro} here beyond Remark \ref{r:F4}. (The classification of one-$W$-type modules, Theorem \ref{t:class}, is obtained for arbitrary unequal parameters though.)

\medskip

We hope that this approach, together with a reduction to the cuspidal setting as in \cite{Be2}, will lead to a better understanding of the \cite{GM} conjecture.

\section{The Dirac operator for rational Cherednik algebra}\label{sec:2}
In this section, we recall from \cite{Ci} the definition of Dirac cohomology for rational Cherednik algebra modules and the basic properties. For the purpose of this paper, we only need to consider the rational Cherednik algebra at $t=0$.

\subsection{Rational Cherednik algebra at $t=0$}
As before, let $\fh$ be a dimensional $\bC$-vector space, denote by $\fh^*$ its dual, and $V=\fh+\fh^*.$ Let $\langle~,~\rangle:V\times V\to\bC$ be the  bilinear symmetric pairing defined by
\begin{equation}
\<x,x\>=0,\ \<y,y\>=0,\ \<x,y\>=\<y,x\>=x(y),
\end{equation}
for all $x\in\fh^*$ and $y\in\fh$. 
Let $W\subset GL(\fh)$ be a complex reflection group with set of pseudo-reflections $\C R$ acting diagonally on $V$. The form $\<~,~\>$ is $W$-invariant.

For every reflection $s\in\C R$, the spaces $\im(\Id_V-s)|_{\fh^*}$ and $\im(\Id_V-s)|_{\fh}$ are one-dimensional. Choose $\al_s$ and $\al_s^\vee$ nonzero elements in $\im(\Id_V-s)|_{\fh^*}$ and $\im(\Id_V-s)|_{\fh}$, respectively. Then there exists $\lambda_s\in \bC$, $\lambda_s\neq 1$ a root of unity, such that
\begin{equation}
s(\al_s^\vee)=\lambda_s \al_s^\vee,\ s(\al_s)=\lambda_s^{-1} \al_s.
\end{equation}
(In the case when $W$ is a finite reflection group, $\lambda_s=-1$.)
For every $v\in V$ such that $\<v,v\>\neq 0$, denote by $s_v$ the reflection in the hyperplane perpendicular to $v$. The reflection $s_v$ is given by:
\[s_v(u)=u-\frac 2{\langle v,v\rangle}\langle u, v\rangle v,\ u\in V.\]
Let $\sqrt{\lambda_s}$ be a square root of $\lambda_s$. Then $s=s_{v_s}s_{v_s'} \in O(V)$, where  $v_s=\sqrt{\lambda_s}\al_s^\vee+\al_s$ and $v_s'=\al_s^\vee+\sqrt{\lambda_s}\al_s$. 

\begin{definition}
The rational Cherednik algebra $\bfH_{0,c}$ associated to $\fh,W$ and the $W$-invariant parameter function $c:\C R\to\bC$ is the quotient of $T(V)\rtimes W$ by the relations:
\begin{enumerate}
\item $[y_1,y_2]=0$, $[x_1,x_2]=0$, for all $y_1,y_2\in \fh$, $x_1,x_2\in\fh^*$;
\item $\displaystyle{[y,x]=-\sum_{s\in\C R} c_s\frac{\langle y,\al_s\rangle\langle\al_s^\vee,x\rangle}{\langle\al_s^\vee,\al_s\rangle} s},$ for all $y\in\fh,$ $x\in \fh^*.$
\end{enumerate}
\end{definition}
Let $\{y_i\}$ be a basis of $\fh$ and $\{x_i\}$ the dual basis of $\fh^*$. 
Define the element
\begin{equation}\label{e:Omega-cherednik}
\Omega_\bfH=2\sum_i x_iy_i-2\sum_{s\in \C R}\frac {c_s}{1-\lambda_s} s\in Z(\bfH_{0,c}).
\end{equation}

\subsection{The Clifford algebra} Let $C(V)$ be the complex Clifford algebra defined by $V$ and $\langle~,~\rangle$.  In terms of the basis $x_i,y_i$'s the relations in $C(V)$ are:
\begin{equation}
x_i\cdot x_j=-x_j\cdot x_i,\ y_i\cdot y_j=-y_j\cdot y_i,\ x_i\cdot y_j+y_j\cdot x_i=-2\delta_{i,j}.
\end{equation} 
Since $V$ is even dimensional, $C(V)$ has a unique complex simple module $S$. The spin module $S$ is realized on the vector space $\bigwedge \fh$ with the action:
\begin{equation}
\begin{aligned}
y\cdot (y_1\wedge \dots\wedge y_k)&=y\wedge y_1\wedge\dots\wedge y_k,\quad y\in \fh;\\
x\cdot (y_1\wedge\dots\wedge y_k)&=2\sum_{i}(-1)^i\langle y_i,x\rangle y_1\wedge\dots\wedge\hat y_i\wedge\dots\wedge y_k.
\end{aligned}
\end{equation}

\subsection{Pin cover of $W$} Following \cite[\S2.3]{Ci}, for every $s\in S$, define
 \begin{equation}\label{e:tau-s}
\tau_s=\frac{1-\lambda_s}{2\langle\al_s^\vee,\al_s\rangle}\al_s\al_s^\vee+1\in C(V).
\end{equation}
By \cite[Lemma 4.6]{Ci}, the map $s\mapsto \tau_s$ extends to a group homomorphism
\begin{equation}\label{l:W-embed}
\tau:W\to C(V)^\times.
\end{equation}
Define $\tau_w$ to be the image in $C(V)^\times$ of $w\in W$ under this map.

Since $C(V)$ acts on $S$, we get an action of $\tau(W)$ on S. 

\begin{lemma}[{\cite[Lemma 4.8]{Ci}}]\label{l:W-action}
The action of $\tau(W)$ on $S$ preserves each piece $\bigwedge^\ell\fh$ of $S$, where it acts by the dual of the natural action, i.e.:
\begin{equation}
\tau_s\cdot(y_1\wedge\dots\wedge y_\ell)={\det}_{\fh}(s)~s(y_1)\wedge\dots\wedge s(y_\ell).
\end{equation} 
\end{lemma}

\subsection{The Dirac element}
The Dirac element in $\bfH_{0,c}\otimes C(V)$ is:
\begin{equation}
\C D=\sum_{i} x_i\otimes y_i+\sum_{i} y_i\otimes x_i.
\end{equation}

We denote by $$\Delta:\bC[W]\to \bfH_{t,c}\otimes C(V)$$ the linear map that extends the assignment $w\mapsto w\otimes \tau_w.$

\begin{proposition}[{\cite[Proposition 4.9]{Ci}}]\label{p:square-dirac}
The Dirac element has the following properties in $\bfH_{0,c}\otimes C(V)$:
\begin{enumerate}
\item  $\C D$ is invariant with respect to the conjugation action of $\Delta(W).$
\item The square equals:
\begin{equation}
\C D^2=-\Omega_\bfH\otimes 1-\Delta(\Omega_{W,c}),
\end{equation}
 with $\Omega_\bfH\in Z(\bfH_{0,c})$  given by (\ref{e:Omega-cherednik}) and 
 \begin{equation}\label{e:Omega-W-2}
 \Omega_{W,c}=\sum_{s\in \C R} \frac {2c_s}{1-\lambda_s} s\in \bC[W]^W.
 \end{equation}
 \end{enumerate}
\end{proposition}

\subsection{Dirac cohomology}
Let $X$ be a finite dimensional $\bfH_{0,c}$-module. The Dirac operator of $X$ (and $\C S$) is 
\begin{equation}
D_X: X\otimes S\to X\otimes S, 
\end{equation}
given by the action of the Dirac element $\C D$. The Dirac cohomology of $X$ (and $\C S$) is 
\begin{equation}
H_D(X)=\ker D_X/\ker D_X\cap \im D_X.
\end{equation}
If nonzero, $H_D(X)$ is a finite dimensional $W$-representation.

The state now the main results about Dirac cohomology applied to this setting. Recall that the center $Z(\bfH_{0,c})$ is nontrivial.

\begin{theorem}[{\cite[Theorem 3.5 and Theorem 3.8]{Ci}}]\label{t:vogan}
For every $z\in Z(\bfH_{0,c})$ there exists a unique element $\zeta_{0,c}(z)\in \bC[W]^W$ and an element $a\in (\bfH_{0,c}\otimes C(V))^W$ such that
\begin{equation}
z\otimes 1=\Delta(\zeta_{0,c}(z))+\C Da+a\C D\text{ in } \bfH_{0,c}\otimes C(V).
\end{equation}
Moreover, the assignment $\zeta_{0,c}: Z(\bfH_{0,c})\to \bC[W]^W$ is an algebra homomorphism.
\end{theorem}
Consider the dual morphism 
\begin{equation}
\zeta_{0,c}^*: \Irr(W)=\Spec\bC[W]^W\to \Spec(Z(\bfH_{0,c})=X_c(W).
\end{equation}
The space $X_c(W)$ in the image of $\zeta_{0,c}^*$ is the generalized Calogero-Moser space \cite{EG}.

\begin{theorem}[{\cite[Theorem 3.14]{Ci}}]\label{t:vogan-conj}
Let $X$ be a finite dimensional $\bfH_{0,c}$-module and assume that $Z(\bfH_{0,c})$ acts on $X$ via the central character $\chi\in \Spec(Z(\bfH_{0,c}))$. Suppose $H_D(X)\neq 0$. If $\sigma\in\Irr(W)$ is such that
\[\Hom_W[\sigma, H_D(X)]\neq 0,
\]
then
\[\chi=\zeta^*_{0,c}(\sigma).
\]
\end{theorem}

 By \cite{EG}, the center $Z(\bfH_{0,c})$ contains the subalgebra $\fk m:=S(\fh)^W\otimes S(\fh^*)^W$ and it is a free $\fk m$-module of rank $|W|$. 
The inclusion $\fk m\subset Z(\bfH_{0,c})$ induces a surjective morphism
\begin{equation}\label{e:Y}
\Upsilon: X_c(W)\to \fh^*/W\times \fh/W.
\end{equation}
Let $\fk m_+$ be the augmentation ideal of $\fk m$ and define similarly $S(\fh)^W_+$ and $S(\fh^*)^W_+$.
Then \cite[Theorem 5.8]{Ci} says that 
the algebra homomorphism from Theorem \ref{t:vogan} factors through $Z(\bfH_{0,c})/\fk m_+$:
\[\zeta_{0,c}: Z(\bfH_{0,c})/\fk m_+\to \bC[W]^W,
\]
 and so the dual morphism is
 \begin{equation}\label{e:CM}
\zeta_{0,c}^*:\Irr(W)\to \Upsilon^{-1}(0).
\end{equation}

Following \cite{Go}, let us consider the  "baby Verma modules" for $\bfH_{0,c}.$ Define
\[\bar\bfH_{0,c}=\bfH_{0,c}/\fk m_+\bfH_{0,c}\]
This is a finite dimensional algebra of dimension $|W|^3$, isomorphic to $S(\fh^*)_W\otimes S(\fh)_W\otimes \bC[W]$ as a vector space, where we denote by $S(\fh)_W=S(\fh)/S(\fh)^W_+$ the graded algebra of coinvariants and similarly for $S(\fh^*)_W$. 

For every $(\sigma,V_\sigma)\in \Irr(W)$, let
\begin{equation}
\bar M(\sigma)=\bar\bfH_{0,c}\otimes_{S(\fh)_W\rtimes \bC[W]}V_\sigma
\end{equation}
be the baby Verma module induced from $\sigma$. Here $S(\fh)_W$ acts by $0$ on $V_\sigma.$ 

\begin{theorem}[{\cite[Proposition 4.3]{Go}}]
\begin{enumerate}
\item For every $\sigma\in\Irr(W)$, the module $\bar M(\sigma)$ is indecomposable and it has a unique simple quotient $\bar L(\sigma)$.
\item The set $\{\bar L(\sigma):\sigma\in\Irr(W)\}$ gives a complete list of non isomorphic simple $\bar\bfH_{0,c}$-modules.
\end{enumerate}
\end{theorem}

As a consequence (\cite[\S5.4]{Go}), the map 
\begin{equation}\label{e:Theta}
\Theta:\Irr(W)\to \Upsilon^{-1}(0)=\Spec Z(\bar\bfH_{0,c}),
\end{equation}
 given by mapping $\bar M(\sigma)$ to its central character, is surjective.

\begin{corollary}[{\cite[Corollary 5.10]{Ci}}]\label{c:zeta=theta} The morphism $\zeta^*_{0,c}:\Irr(W)\to \Upsilon^{-1}(0)$ from (\ref{e:CM}) is the determinant dual of the morphism $\Theta$ from (\ref{e:Theta}), i.e.,
\[\Theta(\sigma)=\zeta_{0,c}^*(\sigma\otimes\det),\text{ for all }\sigma\in \Irr(W).\]
\end{corollary}

\begin{remark}
The partition of $\Irr(W)$ according to the fibers of the map $\Theta$ is called the Calogero-Moser (CM) partition. By Corollary \ref{c:zeta=theta}, we know that this partition is the same as the Dirac partition, i.e., the one given by $\zeta_{0,c}^*.$ This allows us to use Theorem \ref{t:vogan-conj} in order to study the CM partition. 
\end{remark}

\section{Classification of one-W-type modules}\label{sec:3}
In order to apply Theorem \ref{t:vogan-conj}, we construct modules for $\bfH_{0,c}$ for which the Dirac cohomology is easy to compute.

\subsection{Definition} Let $X$ be a simple $\bfH_{0,c}$-module. We consider simple modules $X$ such that the restriction of $X$ to $W$ is an irreducible $W$-representation. In the setting of graded affine Hecke algebras, this type of modules where introduced and classified in \cite{BM}, and studied further in \cite{CM}.

Given an irreducible $W$-representation $(\sigma,U_\sigma)$, the question is how to define the action of $\fk h$ and $\fk h^*$ in such a way that the commutation relations in $\bfH_{0,c}$ are preserved. Suppose $(\pi,X)$ is a one-$W$-type $\bfH_{0,c}$-module extending $\sigma$. As in \cite[Appendix]{CM}, consider the space 
\[
\fh\cdot U_\sigma=\{\pi(y)u: y\in\fh,\ u\in U_\sigma\}.\]
Since $w\cdot y\cdot w^{-1}=w(y)$, the space $\fh\cdot U_\sigma$ is $W$-invariant. Moreover, it is a constituent of the $W$-representation $\refl\otimes \sigma$, where $\refl$ is the reflection representation of $W$ on $\fk h.$ Since the restriction of $X$ to $W$ is irreducible, either $\fh\cdot U_\sigma=0$ or else $\fh\cdot U_\sigma\cong U_\sigma$ as $W$-representations. The latter case implies that
\begin{equation}\label{e:cond}
\Hom_{W}[\sigma,\sigma\otimes\refl]\neq 0.
\end{equation}
A similar analysis applies to $\fh^*\cdot U_\sigma.$

\begin{remark}\label{r:CiMo}
When $W$ is a finite Weyl group, \cite[Appendix]{CM} determines all $W$-types $\sigma$ that satisfy condition (\ref{e:cond}). When the long Weyl group element $w_0$ is central in $W$, an easy argument implies that no such $\sigma$ exists. When $w_0$ is not central, i.e., types $A_{n-1}$, $D_{2n+1}$, or $E_6$, there exist representations $\sigma$ satisfying (\ref{e:cond}). However, in the setting of the rational Cherednik algebra (unlike that of the graded affine Hecke algebra) we already know that for example in type $A$, no one-$W$-type modules can exist \cite{EG}.
\end{remark}

From now on, we assume that $W$ is a finite Weyl group. In light of Remark \ref{r:CiMo} and the discussion preceding it, we make the following definition.

\begin{definition}\label{d:one-type}
A simple $\bfH_{0,c}$-module $(\pi,X)$ is called a one-$W$-type module if 
\begin{enumerate}
\item the restriction of $X$ to $W$ is irreducible, and 
\item $\pi(y)=0=\pi(x)$ for all $y\in\fh$, $x\in\fh^*.$
\end{enumerate}
Condition (2) is automatic when $w_0$ is central in $W$.
\end{definition}

\subsection{Criterion} We determine which irreducible $W$-representations give rise to one-$W$-type modules in the sense of Definition \ref{d:one-type}. Let $R\subset \fh^*$ denote the set of roots, $R^+$, the set of positive roots, $\Pi$ the set of simple roots. Let $R^\vee$, $R^{+,\vee}$, and $\Pi^\vee$ denote the corresponding coroots in $\fh.$ If $\al\in R$ is a root, let $\al^\vee$ be the corresponding coroot normalized such that $\langle\al,\al^\vee\rangle=2.$ The (renormalized) basic commutation relation in $\bfH_{0,c}$ becomes
\begin{equation}\label{e:comm}
[y,x]=-\sum_{\al>0} c_\al\langle y,\alpha\rangle\langle\al^\vee,x\rangle s_\al,\quad y\in\fh,\ x\in\fh^*.
\end{equation}

\begin{lemma}
An irreducible $W$-representation $\sigma$ extends to a one-$W$-type $\bfH_{0,c}$-module if and only if 
\begin{equation}\label{e:crit}
\sum_{\al>0} c_\al \langle y,\alpha\rangle\langle\al^\vee,x\rangle \sigma(s_\al)=0,
\end{equation}
for all $y\in\fh$ and $x\in\fh^*.$
\end{lemma}

\begin{proof}
This is straightforward by Definition \ref{d:one-type} and formula (\ref{e:comm}).
\end{proof}

In order to carry out the computations effectively, we adapt a reduction criterion from \cite[Proposition 2.4]{BM}. 

\begin{proposition}\label{p:reduction} Suppose $y\in \fh$ and $x_1,x_2\in\fh^*$ are such that 
\begin{enumerate}
\item[(a)] $W\cdot y$ spans $\fh$ (modulo the $W$-invariants), and
\item[(b)] $\{Z_W(y)\cdot x_1,x_2\}$ span $\fh^*$.
\end{enumerate}
The $W$-type $\sigma$ extends to a one-$W$-type module $(\pi,X)$ if and only if equation (\ref{e:crit}) holds for the pairs $(y,x_1)$ and $(y,x_2)$.
\end{proposition}

\begin{proof}
Without loss of generality, assume $R$ spans $\fh^*$. Suppose $\pi([y,x_1])=0$ and $\pi([y,x_2])=0.$ Let $y'\in\fh$ and $x'\in fh^*$ be arbitrary and we want to show that $\pi([y',x'])=0$. For every $w\in W$, we have 
\[ [w(y'),w(x')]=w\cdot [y',x']\cdot w^{-1}.\]
Since $\pi([y,x_1])=0$ it follows from this that $\pi([y,w(x_1)])=0$ for all $w\in Z_W(y)$. From (b), it follows that $\pi([y,x'])=0$ for all $x'\in\fh^*.$ But then also $\pi([w(y),x'])=0$ for all $w\in W$ and $x'\in\fh^*$, and the conclusion follows from (a).
\end{proof}

\subsection{Classification}
To describe the classification result, let us recall first the parameterization of irreducible $W$-representations. If $W=S_n$ is of type $A_{n-1}$, the $W$-types are parameterized by partitions $\lambda$ of $n$; we write $(\lambda)$ for the corresponding representation. In this notation, $(n)$ denotes the trivial representation and $(1^n)$ the sign representation. We also denote by $\lambda^t$ the transpose partition to $\lambda.$

If $W$ is of type $B_n$ (equivalently $C_n$), then $W_n=S_n\ltimes (\bZ/2\bZ)^n$ and the irreducible $W$-representations are parameterized by bipartitions $(\lambda,\mu)$ of $n$ via Mackey induction. We denote by $(\lambda)\times (\mu)$ the $W$-type
\[(\lambda)\times(\mu)=\Ind_{S_{n-k}\times S_k\times (\bZ/2\bZ)^n}^{S_n}((\lambda)\otimes (\mu)\otimes \triv^{\otimes(n-k)}\otimes\sgn^{\otimes k}).
\]
If $R$ is of type $D_n$, $W(D_n)$ is an index 2 subgroup of $W(B_n)$ and the irreducible $W(D_n)$-representations are obtained by restriction from $W(B_n)$. If $\lambda\neq\mu$, then $(\lambda)\times (\mu)$ and $(\mu)\times (\lambda)$ restrict to the same $W(D_n)$-representation, which we denote again by $(\lambda)\times(\mu).$ When $\lambda=\mu$, the restriction of the $W(B_n)$-type $(\lambda)\times (\lambda)$ to $W(D_n)$ splits into a sum of two equidimensional, non isomorphic representations, denoted $(\lambda)\times (\lambda)_I$ and $(\lambda)\times (\lambda)_{II}$. 

When $W$ is of exceptional type, we will use the notation of Kondo for $F_4$ and Frame for $E_6,E_7,E_8$, see \cite{Lu}, and Carter's notation for $G_2$ \cite{Ca}. 

\smallskip

If the parameter function $c$ for the rational Cherednik algebra $\bfH_{0,c}$ is identically zero, then every $W$-type extends trivially to an $\bfH_{0,c}$-module. We assume this is not the case from now on. Thus, without loss of generality, we may assume $c$ is identically $1$ if the root system is simple and simply-laced. If the simple root system $R$ has two root lengths, we denote by $c_s$ the parameter on the short roots and by $c_\ell$ the parameter on the long roots. When the algebra if of type $B/C$, this convention refers to the type $B_n$ root system.

\begin{theorem}\label{t:class} Let $R$ be a simple root system. The irreducible $W$-representations that extend to a one-$W$-type $\bH_{0,c}$-module in the sense of Definition \ref{d:one-type} are:

\begin{enumerate}
\item $\mathbf{A_{n-1}}$ or $\mathbf{E_7}$: none;
\item $\mathbf{B_n}$: the $W$-representations of the form $(\lambda)\times (0)$ and $(0)\times (\lambda^t)$, where $\lambda$ is  a rectangular partition of $n$ of the form $\lambda=(\underbrace{d,d,\dots,d}_{k})$ with $$k-d=c_s/c_\ell;$$
\item $\mathbf{D_n}$: the $W$-representations of the form $(\lambda)\times (0)$ where $\lambda$ is  a rectangular partition of $n$ of the form $\lambda=(\underbrace{d,d,\dots,d}_{d})$ (so $n=d^2$);
\item $\mathbf {E_6}$: the representation $10_s$;
\item $\mathbf{E_8}$: the representations $168_y$ and $420_y$;
\item $\mathbf{G_2}$: 
\begin{enumerate}
\item $\phi_{1,3}'$, $\phi_{1,3}''$, and $\phi_{2,2}$ when $c_s/c_\ell=1$;
\item $\phi_{1,0}$, $\phi_{1,6}$, and $\phi_{2,1}$ when $c_s/c_\ell=-1$;
\end{enumerate}
\item $\mathbf{F_4}$: 
\begin{enumerate}
\item  $4_1$ for all values of the parameters $c_s,c_\ell$;
\item  $1_2,1_3,6_1,4_3,4_4$ when $c_s/c_\ell=1$;
\item  $1_1,1_4,6_2,4_2,4_5$ when $c_s/c_\ell=-1$;
\item $2_1,2_2$ when $c_\ell=0$;
\item  $2_3,2_4$ when $c_s=0$;
\item $12_1, 16_1$ when $c_s/c_\ell=\pm\sqrt{-1}$.
\end{enumerate}

\end{enumerate}

\end{theorem}

In the rest of the section, we present the proof of Theorem \ref{t:class}.

\subsection{Classical types} We write the roots of root system of classical types in the standard coordinates. For simplicity, we work with the algebra of type $gl(n)$ in type $A$. Let $\{\ep_1,\dots,\ep_n\}$ be the standard coordinates for $\fh^*$ and 
$\{\ep_1^\vee,\dots,\ep_n^\vee\}$ the dual coordinates for $\fh.$ The positive roots are denoted as follows:
\begin{enumerate}
\item $\mathbf {A_{n-1}}$: $\ep_i-\ep_j$, $1\le i<j\le n$;
\item $\mathbf {B_n}$: $\ep_i\pm \ep_j$, $1\le i<j\le n$, and $\ep_i$, $1\le i\le n$;
\item $\mathbf{D_n}$: $\ep_i\pm\ep_j$, $1\le i<j\le n$. 
\end{enumerate}
In the notation of Proposition \ref{p:reduction}, we choose in all cases
\begin{equation}
y=\ep_1^\vee,\quad x_1=\ep_1,\quad x_2=\ep_2.
\end{equation}
Suppose $\sigma$ is a $W$-type that can be extended to a one-$W$-type module. The two conditions from Proposition \ref{p:reduction} imply in type $A_{n-1}$:
\begin{align}
&\sum_{i=2}^n \sigma(s_{\ep_1-\ep_i})=0\text{ and }\sigma(s_{\ep_1-\ep_2})=0.
\end{align}
But clearly the second condition is impossible, so there are no such $S_n$-representations.

\medskip

Suppose that we are in type $B_n$ now. The two conditions become
\begin{equation}
\begin{aligned}
&c_\ell\sum_{i=1}^n(\sigma(s_{\ep_1-\ep_i})+\sigma(s_{\ep_1+\ep_i}))+2 c_s \sigma(s_{\ep_1})=0;\\
&c_\ell(-\sigma(s_{\ep_1-\ep_2})+\sigma(s_{\ep_1+\ep_2}))=0.
\end{aligned}
\end{equation}
If $c_\ell=0$, the two conditions imply that $c_s\sigma(s_{\ep_1})=0$, which is impossible unless $c_s=0$. So suppose $c_\ell\neq 0$, and set $c=c_s/c_\ell$. The second condition implies that $\sigma(s_{\ep_1-\ep_2})=\sigma(s_{\ep_1+\ep_2})$ and by conjugation that \[\sigma(s_{\ep_i-\ep_j})=\sigma(s_{\ep_i+\ep_j})=\sigma(s_{\ep_j})\sigma(s_{\ep_i-\ep_j})\sigma(s_{\ep_j}),\] for all $i<j$. But this means that every $\sigma(s_{\ep_j})$ preserves with the restriction of $\sigma$ to $S_n$, and since $\sigma$ is irreducible, it must be of the form $(\lambda)\times (0)$ or $(0)\times (\lambda^t).$

The two cases are dual to each other via tensoring with the sign representation, so it is sufficient to treat the first one. If $\sigma=(\lambda)\times (0)$, then $\sigma(s_{\ep_i})=\Id$ for all $i$ and $\sigma(s_{\ep_i-\ep_j})=(\lambda)(s_{\ep_i-\ep_j}).$ The condition is then one in $S_n$:
\begin{equation}\label{e:harmonic}
\sum_{i=2}^n (\lambda)(s_{\ep_1-\ep_i})+c\cdot \Id=0.
\end{equation}
A similar condition appears in \cite{BM} in the setting of the graded Hecke algebra of type $A$, see the proof of Theorem 3.15, particularly (3.18) in \cite{BM}. One can compute it using the realization of $S_n$-representations in harmonic polynomials and the result is that (\ref{e:harmonic}) holds if and only if $\lambda$ is a rectangular partition $\lambda=(\underbrace{d,\dots,d}_k)$ satisfying
\[c=k-d.\]

\medskip

In type $D_n$, the two necessary and sufficient conditions are:
\begin{equation}
\begin{aligned}
&\sum_{i=1}^n(\sigma(s_{\ep_1-\ep_i})+\sigma(s_{\ep_1+\ep_i}))=0;\\
&-\sigma(s_{\ep_1-\ep_2})+\sigma(s_{\ep_1+\ep_2})=0.
\end{aligned}
\end{equation}
The discussion in this case is identical with the case $B_n$ with $c_\ell=1$ and $c_s=0.$

\subsection{Exceptional types} For the exceptional simple root systems, we would like to use characters of representations (rather than the representations themselves), so we need to adapt Proposition \ref{p:reduction}. The idea is the same as in \cite[Proposition 4.1]{BM}. If $\sigma$ is a $W$-representation, let $\chi_\sigma$ denote its character. 

\begin{proposition}\label{p:reduction-2}
Let $\sigma$ be an irreducible $W$-representation and $y,x_1,x_2$ as in Proposition \ref{p:reduction}. Then $\sigma$ extends to a one-$W$-type $\bfH_{0,c}$-module if and only if
\[\chi_{\sigma}([y,x_i]^2)=0,\quad i=1,2.\]
\end{proposition}

\begin{proof}
We need to show that $\sigma([y,x])=0$ if and only if $\chi_\sigma([y,x]^2)=0.$ We can realize $\sigma(s_\al)$ as a symmetric real-valued (in fact, rational) matrix for every $\al>0.$ This means that $\sigma([y,x])$ is also a real-valued symmetric matrix and therefore, it is diagonalizable and has real eigenvalues $\nu_i$, $i=1,\dim\sigma$. But then $\chi_\sigma([y,x]^2)=\sum_i \nu_i^2=0$ if and only if $\nu_i=0$ for all $i$. 
\end{proof}

The explicit calculations are as follows. For each exceptional simple root system, we choose $y,x_1,x_2$ and compute $[y,x_i]$, $i=1,2$, as combination of $s_\al$, $\al>0$. Next we compute the squares $[y,x_i]^2$ which are linear combination of terms $s_\al s_\beta$. For the character computation in Proposition \ref{p:reduction-2}, the only important feature is which rank $2$ root subsystem every such pair $\al,\beta$ forms. Thus we need to count all the occurrences of a given rank $2$ subsystem in this combination. We used the computer algebra systems ''Mathematica'' to carry out these straightforward, but tedious root calculations and  GAP version 3.4 \cite{GAP} with the package 'chevie' \cite{GHLMP} to compute the characters. 

We summarize the results next. We will denote by $w_{A_2}$, $w_{\wti A_2}$, $w_{2A_1}$, $w_{A_1+\wti A_1}$, $w_{B_2}$, $w_{G_2}$ representatives of the conjugacy classes of the corresponding rank $2$ root subsystems.

\medskip

Let $R$ be of type $G_2$. We choose the coordinates for the simple roots:
\begin{equation}
\al_s=(2/3,-1/3,-1/3),\quad \al_\ell=(-1,1,0),
\end{equation}
and 
\begin{equation}
y=(1,1,-2),\quad x_1=(1/3,1/3,-2/3),\quad x_2=(0,1,-1).
\end{equation}
Then $[y,x_1]$ and $[y,x_2]$ have $5$ and $4$ terms, respectively. Calculating the quantities in Proposition \ref{p:reduction-2}, we find that $\sigma$ extends to a one-$W$-type module if and only if $\chi_\sigma$ vanishes on
\begin{equation}
\begin{aligned}
&3(c_\ell^2+c_s^2) 1+(2 c_\ell c_s) w_{A_1+\wti A_1}+ 3(c_\ell^2+c_s^2) w_{A_2}+ (10 c_\ell c_s) w_{G_2},\text{ and}\\
&5(c_\ell^2+c_s^2) 1+(2 c_\ell c_s) w_{A_1+\wti A_1}+ 4(c_\ell^2+c_s^2) w_{A_2}+ (16 c_\ell c_s) w_{G_2}.
\end{aligned}
\end{equation}
Using the character table for $G_2$ (in GAP), we verify Theorem \ref{t:class} for $G_2$.

\medskip

Let $R$ be of type $F_4$. We choose the following coordinates for the simple roots:
\begin{equation}
\al_1=\ep_1-\ep_2-\ep_3-\ep_4,\ \al_2=2\ep_4,\ \al_3=\ep_3-\ep_4,\ \al_4=\ep_2-\ep_3
\end{equation}
and 
\begin{equation}
y=\ep_1^\vee,\quad x_1=\ep_1,\quad x_2=\ep_1+\ep_2+\ep_3+\ep_4.
\end{equation}
Then $[y,x_1]$ has $15$ terms and $[y,x_2]$ has $9$ terms. Calculating the squares, we find that $\sigma$ extends to a one-$W$-type module if and only if $\chi_\sigma$ vanishes on
\begin{equation}
\begin{aligned}
&(c_\ell^2+c_s^2)1+(c_\ell^2+c_s^2) w_{2A_1}+(4c_\ell c_s) w_{A_1+\wti A_1}+(4 c_\ell^2) w_{A_2}+(4 c_s^2) w_{\wti A_2}+(8 c_\ell c_s) w_{B_2};\\
&(c_\ell^2+c_s^2)1+(2 c_\ell^2) w_{A_2}+ (2 c_s^2) w_{\wti A_2}+ (6c_\ell c_s) w_{B_2}.
\end{aligned}
\end{equation}
Using GAP, we extract the character table of $W(F_4)$ for these conjugacy classes, and verify Theorem \ref{t:class} in this case.

\medskip

Let $R$ be of type $E_6$. We assume the parameters equal to $1$. The simple roots are
\begin{equation}
\begin{aligned}
&\al_1=\frac 12(1,-1,-1,-1,-1,-1,-1,1),\ \al_2=\ep_1+\ep_2, \al_3=-\ep_1+\ep_2,\\
&\al_4=-\ep_2+\ep_3,\ \al_5=-\ep_3+\ep_4,\ \al_6=-\ep_4+\ep_5.
\end{aligned}
\end{equation}
Let $\om_i,\om_i^\vee$ be the fundamental weights, respectively coweights. We choose
\begin{equation}
y=\om^\vee_2,\quad x_1=\om_2,\quad x_2=s_{\al_2}(x_1).
\end{equation}
Then $[y,x_1]$ has $21$ terms and $[y,x_2]$ has $12$ terms. Calculating the squares, we find that $\sigma$ extends to a one-$W$-type module if and only if $\chi_\sigma$ vanishes on
\begin{equation}
\begin{aligned}
1+5 w_{2A_1}+10 w_{A_2}\text{ and } 1+w_{2A_1}+6w_{A_2}.
\end{aligned}
\end{equation}
Using GAP, we extract the character table of $W(E_6)$ for the conjugacy classes $1$, $2A_1$ and $A_2$, and find that the only $W$-type whose character vanishes is $10_s$.

\medskip

Let $R$ be of type $E_7$. We assume the parameters equal to $1$. The simple roots are
\begin{equation}
\begin{aligned}
&\al_1=\frac 12(1,-1,-1,-1,-1,-1,-1,1),\ \al_2=\ep_1+\ep_2, \al_3=-\ep_1+\ep_2,\\
&\al_4=-\ep_2+\ep_3,\ \al_5=-\ep_3+\ep_4,\ \al_6=-\ep_4+\ep_5,\ \al_7=-\ep_5+\ep_6.
\end{aligned}
\end{equation}
We choose
\begin{equation}
y=\om^\vee_1,\quad x_1=\om_1,\quad x_2=s_{\al_1}(x_1).
\end{equation}
Then $[y,x_1]$ has $33$ terms and $[y,x_2]$ has $18$ terms. Calculating the squares, we find that $\sigma$ extends to a one-$W$-type module if and only if $\chi_\sigma$ vanishes on
\begin{equation}
\begin{aligned}
1+10 w_{2A_1}+16 w_{A_2}\text{ and } 2+5w_{2A_1}+20w_{A_2}.
\end{aligned}
\end{equation}
Using GAP, we extract the character table of $W(E_7)$ for the conjugacy classes $1$, $2A_1$ and $A_2$, and find that no $W$-characters vanish on these elements.

\medskip

Let $R$ be of type $E_8$. We assume the parameters equal to $1$. The simple roots are
\begin{equation}
\begin{aligned}
&\al_1=\frac 12(1,-1,-1,-1,-1,-1,-1,1),\ \al_2=\ep_1+\ep_2, \al_3=-\ep_1+\ep_2,\\
&\al_4=-\ep_2+\ep_3,\ \al_5=-\ep_3+\ep_4,\ \al_6=-\ep_4+\ep_5,\ \al_7=-\ep_5+\ep_6\ \al_8=-\ep_6+\ep_7.
\end{aligned}
\end{equation}
We choose
\begin{equation}
y=\om^\vee_8,\quad x_1=\om_8,\quad x_2=s_{\al_8}(x_1).
\end{equation}
Then $[y,x_1]$ has $57$ terms and $[y,x_2]$ has $30$ terms. Calculating the squares, we find that $\sigma$ extends to a one-$W$-type module if and only if $\chi_\sigma$ vanishes on
\begin{equation}
\begin{aligned}
1+21 w_{2A_1}+28 w_{A_2}\text{ and } 1+6w_{2A_1}+18w_{A_2}.
\end{aligned}
\end{equation}
Using GAP, we extract the character table of $W(E_8)$ for the conjugacy classes $1$, $2A_1$ and $A_2$, and find that the only $W$-characters that vanish on these elements are $168_y$ and $420_y$.

\section{Dirac cohomology and cuspidal two-sided cells}\label{sec:4}
We retain the notation from sections \ref{sec:2} and \ref{sec:3}.

\subsection{One-$W$-type modules} The Dirac cohomology of a one-$W$-type module can be computed easily.

\begin{proposition}\label{p:one-dirac}
Let $(\pi,X)$ be a one-$W$-type $\bfH_{0,c}$-module that extends the irreducible $W$-representation $\sigma$. Then $H_D(X)=\sigma\otimes {\bigwedge}\fh.$
\end{proposition}

\begin{proof}
Since $X$ is a one-$W$-type module, every $x\in\fh^*$ and $y\in \fh$ act by $0$ on $X$. Since $\C D=\sum_{i}x_i\otimes y_i+\sum_i y_i\otimes x_i,$ it follows that $\C D$ also acts by $0$ on $X\otimes S$, and so $\ker D_X=X\otimes S$ while $\im D_X=0$. Thus $H_D(X)=X\otimes S$ as $W$-representations and the claim follows.
\end{proof}

Recall the partition of $\Irr W$ into Calogero-Moser cells, i.e., the fibers of the map $\Theta$ from (\ref{e:Theta}). If $\sigma,\sigma'\in \Irr W$, denote by
\begin{equation}
\sigma\sim_{CM}\sigma' \text{ if }\Theta(\sigma)=\Theta(\sigma')\in \Upsilon^{-1}(0).
\end{equation}
Denote also by $[\sigma]_{CM}$ the Calogero-Moser cell containing $\sigma.$

\begin{corollary}\label{c:one-dirac}
Let $\sigma$ be an irreducible $W$-representation that can be extended to a one-$W$-type module. If $\sigma'\in \Irr W$ is such that $\Hom_W[\sigma',\sigma\otimes{\bigwedge}\fh]\neq 0$, then $\sigma\sim_{CM}\sigma'.$
\end{corollary}

\begin{proof}
If $\Hom_W[\sigma',\sigma\otimes{\bigwedge}\fh]\neq 0$, then $\Hom_W[\sigma'\otimes\sgn,\sigma\otimes{\bigwedge}\fh]\neq 0$, and it follows from Proposition \ref{p:one-dirac} that $\Hom_{W}[\sigma'\otimes\sgn,H_D(X)]\neq 0$. (In particular, this applies for $\sigma$ too.) By Theorem \ref{t:vogan-conj}, $\zeta_{0,c}^*(\sigma\otimes\sgn)=\zeta_{0,c}^*(\sigma'\otimes\sgn)$. Corollary \ref{c:zeta=theta} then says that $\Theta(\sigma)=\Theta(\sigma').$ 
\end{proof}

\subsection{Cuspidal two-sided cells}
We restrict to the case when the parameters $c$ are equal to $1$ and recall the necessary definitions from \cite{Lu}. 

If $\sigma\in \Irr(W)$, let $a(\sigma)\ge 0$ denote the (integer) exponent of the smallest power of $q$ that appears in expansion of the formal degree of the Hecke algebra representation corresponding to $\sigma$. (See \cite[(4.1.1) and \S3.1]{Lu}.)  Let $b(\sigma)$ denote the smallest nonnegative integer such that $\sigma$ occurs in the $i$-th symmetric power of the reflection representation. A result of Lusztig says that $a(\sigma)\le b(\sigma)$ always, and a representation $\sigma$ is called special if equality holds.

Before discussing the partition of $\Irr W$ into double cells, we record the following remarkable result that will be needed later in connection with the formula for $\C D^2$ from Proposition \ref{p:square-dirac}. This fact was noticed empirically by Beynon-Lusztig \cite{BL} and proved uniformly by Opdam \cite{Op}.

\begin{theorem}[{\cite{BL},\cite{Op}}]\label{t:scalar}
Let $W$ be a finite reflection group and $\sigma$ an irreducible $W$-representation. The element $\Omega_{W,1}=\sum_{\al>0}s_\al$ acts in $\sigma$ by
\begin{equation}
N(\sigma)=a(\sigma\otimes\sgn)-a(\sigma).
\end{equation} 
\end{theorem}

If $W'$ is a parabolic subgroup of $W$ and $\mu\in \Irr W'$, define the truncated induction (\cite[(4.1.7)]{Lu}):
\begin{equation}
J_{W'}^W(\mu)=\sum_\sigma\langle\mu,\sigma\rangle_{W'} \sigma,
\end{equation}
where the sum ranges over all $\sigma\in \Irr W$ such that $a(\sigma)=a(\mu).$ (If $\langle\mu,\sigma\rangle_{W'}\neq 0$, then necessarily $a(\sigma)\ge a(\mu).$)

Following \cite[\S4.2]{Lu}, define inductively a partition of $\Irr W$ into families (or two-sided cells) of representations. If $W=\{1\}$, then there is only one family consisting of the trivial representation. Assume that $W\neq \{1\}$ and that we have defined families for all parabolic subgroups $W'\neq W$. Then two $W$-representations $\sigma$ and $\sigma'$ are in the same family for $W$ if there exists a sequence of $W$-representations
\[\sigma=\sigma_0,\sigma_1,\dots,\sigma_m=\sigma',\]
such that for every $i$, $0\le i\le m-1$ there exists a parabolic subgroup $W_i\neq W$ and $W_i$-representations $\mu_i',\mu_i''$ in the same family for $W_i$ such that either
\begin{equation}
\langle \mu_i',\sigma_{i-1}\rangle_{W_i}\neq 0,\ a(\mu_i')=a(\sigma_{i-1}) \text{ and }\langle \mu_i'',\sigma_{i}\rangle_{W_i}\neq 0,\ a(\mu_i'')=a(\sigma_{i}),
\end{equation}
or 
\begin{equation}
\begin{aligned}
&\langle \mu_i',\sigma_{i-1}\otimes\sgn\rangle_{W_i}\neq 0,\ a(\mu_i')=a(\sigma_{i-1}\otimes\sgn) \text{ and }\\&\langle \mu_i'',\sigma_{i}\otimes\sgn\rangle_{W_i}\neq 0,\ a(\mu_i'')=a(\sigma_{i}\otimes \sgn).
\end{aligned}
\end{equation}
From the definition one sees that if $\C F$ is a family in $\Irr W$ then so is $\C F\otimes\sgn.$ Moreover, from Theorem \ref{t:scalar}, we see that the scalars $N(\sigma)$ are constant as $\sigma$ varies in a fixed family.

The description of the families for each simple Weyl group is made explicit in \cite[Chapter 4]{Lu}. In the case of classical groups other than type $A$, the characterization of families is in terms of certain symbols, see \cite[\S4.5-\S4.7]{Lu}

We also recall that each family contains a unique special representation.

\begin{definition}[{\cite[\S8.1]{Lu}}] A family $\C F\subset \Irr W$ is called non-cuspidal if there exists a proper parabolic subgroup $W'$ of $W$ and a family $\C F'\subset \Irr W'$ such that either
\begin{enumerate}
\item[(a)] $J_{W'}^W$ establishes a bijection between $\C F'$ and $\C F$ or
\item[(b)] $J_{W'}^W$ establishes a bijection between $\C F'$ and $\C F\otimes\sgn$.
\end{enumerate}
A family $\C F\subset \Irr W$ is called cuspidal if it is not non-cuspidal.
\end{definition}

\begin{theorem}[{\cite[\S8]{Lu}}] The classification of cuspidal families is as follows.
\begin{enumerate}
\item Type $\mathbf{A_{n-1}}$: no cuspidal families.
\item Type $\mathbf{B_n}$: only when $n=d^2+d$ for some $d\ge 1$ and then there is a unique cuspidal family for which the special representation is in bipartition notation $(1,2,\dots,d)\times (1,2,\dots,d)$ and the symbol is $\left(\begin{matrix}0&&2&&4&\dots&&2d\\&1&&3&&\dots &2d-1\end{matrix}\right)$.
\item Type $\mathbf{D_n}$: only when $n=d^2$ for some $d\ge 2$ in which case there is a unique cuspidal family for which the special representation is  in bipartition notation $(1,2,\dots,d)\times (1,2,\dots,d-1)$ and the symbol is $\left(\begin{matrix}0&2&4&\dots&2d-2\\1&3&5&\dots &2d-1\end{matrix}\right)$.
\item Type $\mathbf{G_2}$: the family $\{\phi_{2,1},\phi_{2,2},\phi_{1,3}',\phi_{1,3}''\}$;
\item Type $\mathbf{F_4}$: the family $\{12_1,9_2,9_3,1_2,1_3,4_1,4_3,4_4,6_1,6_2,16_1\}$;
\item Type $\mathbf{E_6}$: the family $\{80_s,60_s,90_s,10_s,20_s\}$;
\item Type $\mathbf{E_7}$: the family $\{512_a',512_a\}$;
\item Type $\mathbf{E_8}$: the family $\{4480_y,3150_y,4200_y,4536_y,5670_y,420_y, 1134_y,1400_y,$ $2688_y,1680_y,168_y, 70_y,7168_w,1344_w,2016_w,5600_w,448_w\}.$
\end{enumerate}
For the exceptional types, the first representation listed in the family is the special one.
\end{theorem}

Denote by $\C F_{\cusp}(W)$ the unique cuspidal family of $W$ if this exists. (Otherwise, $\C F_{\cusp}(W)=\emptyset.$)

\subsection{Cuspidal Calogero-Moser points} Assume the parameter $c$ of the rational Cherednik algebra is identically equal to $1$ (equivalently, any constant). Denote by $\bfH_{0,1}$ this equal parameter algebra. 
If $\bfH_{0,1}$ has one-$W$-type modules, then in these modules every $x\in\fh^*$ and $y\in \fh$ act by zero. Let 
\begin{equation}
\mathbf 0\in \Upsilon^{-1}(0)\subset X_c(W)
\end{equation}
be the corresponding point in the Calogero-Moser space.  We now combine Corollary \ref{c:one-dirac} with the classification of one-$W$-type modules from Theorem \ref{t:class}.

\begin{theorem} \label{t:CM-cusp} Let $\bfH_{0,1}$ be the equal parameter rational Cherednik algebra (at $t=0$) for a simple finite Weyl group $W$.
\begin{enumerate}
\item If $W$ is of type $B_n$ and $n=d^2+d$, then $\Theta^{-1}(\mathbf 0)\supseteq \C F_\cusp(B_n)$.
\item If $W$ is of type $D_n$ and $n=d^2$, then $\Theta^{-1}(\mathbf 0)\supseteq \C F_\cusp(D_n)$.
\item If $W$ is of type $G_2,$ $F_4$, or $E_6$, then $\Theta^{-1}(\mathbf 0)=\C F_\cusp(W).$
\item If $W$ is of type $E_8$, then $\C F_{\cusp}(W)\setminus \{4480_y\}\subseteq \Theta^{-1}(\mathbf 0)\subseteq\C F_\cusp(W)\cup\{2100_y\}.$
\end{enumerate}
\end{theorem}

\begin{proof}
Let $\sigma$ be an irreducible $W$-representation that affords a one-$W$-type module. Then Corollary \ref{c:one-dirac} implies that every irreducible constituent of $\sigma\otimes\bigwedge\fh$ is contained in $\Theta^{-1}(\mathbf 0).$ We discuss each case as in Theorem \ref{t:class}, when the parameters are equal.

Type $\mathbf{B_n}$. Let $n=d^2+d$ consider $\sigma=(\underbrace{d,\dots,d}_{d+1})\times (0)$. The other case is just $\sigma\otimes\sgn$ so it gives the same result. We need to decompose $\sigma\otimes \bigwedge\fh.$ The wedge representations in type $B_n$ are \[{\bigwedge}^\ell\fh=(n-\ell)\times (\ell)=\Ind_{S_{n-\ell}\times S_\ell\times (\bZ/2\bZ)^n}^{W(B_n)}((n-\ell)\otimes (\ell)\otimes 1^{\otimes(n-\ell)}\otimes \ep^{\otimes\ell}),\]
where we denote by $1$ and $\ep$ the trivial and sign $\bZ/2\bZ$-representations, respectively. Then
\[(\lambda)\times(0)\otimes {\bigwedge}^\ell\fh=\Ind_{S_{n-\ell}\times S_\ell\times (\bZ/2\bZ)^n}^{W(B_n)}((\lambda)|_{S_{n-\ell}\times S_\ell}\otimes 1^{\otimes(n-\ell)}\otimes \ep^{\otimes\ell})
\]
If we vary $\ell$ it follows that
\begin{equation}
(\lambda)\times(0)\otimes {\bigwedge}\fh=\sum_{\ell=0}^n \sum_{\mu,\nu} c_{\mu,\nu}^\lambda (\mu)\times(\nu^t),
\end{equation}
where the interior sum is over all partitions $\mu$ of $n-\ell$ and $\nu$ of $\ell$, and $c_{\mu,\nu}^\lambda$ is the Littlewood-Richardson coefficient.

Since $\lambda$ is a rectangular partition, it follows from the Littlewood-Richardson rule that if $\nu=\lambda\setminus \mu$, then $c_{\mu,\nu}^\lambda\neq 0$. Here we regard $\mu$ as a left justified decreasing partition sharing the same upper left corner as $\lambda$ and $\nu$ is the left justified decreasing partition obtained by rotating by $180^\circ$ the complement of $\mu$ in $\lambda$.

We claim that every bipartition $(\mu)\times (\nu^t)$ where $\nu=\lambda\setminus\mu$ is in $\C F_\cusp(B_n).$ Write $\mu$ in the form
\[\mu=(\underbrace{0,\dots,0}_{i_0},\underbrace{1,\dots,1}_{i_1},\dots,\underbrace{d,\dots,d}_{i_d})\]
where $i_j\ge 0$ and $i_0+i_1+\dots+i_d=d+1$. Then $\nu^t=(b_1,b_2,\dots,b_d)$, where
\[b_j=i_0+i_1+\dots+i_{j-1}.\]
The symbol of $(\mu)\times (\nu^t)$ is obtained by adding in order $0,1,\dots,d$ to the entries of $\mu$ (for the first row of the symbol) and $0,1,\dots,d-1$ to the entries of $\nu^t$ (for the second row). It follows that the first row of the resulting symbol is
\[0,1,\dots, b_1-1,\ b_1+1,b_1+2,\dots,b_2,\ b_2+2,b_2+3,\dots,b_3+1,\dots, b_d+(d-2),\ b_d+d\dots, 2d ,\]
while the second row is
\[b_1,b_2+1,b_3+2,\dots, b_d+(d-1).\]
Clearly the two rows are disjoint and their union equals $\{0,1,2,\dots,2d\}$. By the characterization of families in type $B_n$, these are precisely all the $W$-types in $\C F_\cusp(B_n).$

\medskip

Type $\mathbf{D_n}$. Here $n=d^2$ and $\sigma=(\underbrace{d,\dots,d}_{d})\times (0)$. The proof is completely analogous to type $B_n$ and we skip the details.

\medskip

For the exceptional cases, we compute directly using GAP, the decomposition of the tensor product $\sigma\otimes\bigwedge\fh.$

\smallskip

Type $\mathbf{G_2}$. The relevant $\sigma$ are $\phi_{1,3}'$, $\phi_{1,3}''$, and $\phi_{2,2}$. The wedge representations are (in order): $\phi_{1,0}$, $\phi_{2,1}$, and $\phi_{1,6}$. We find
\begin{equation}
\begin{aligned}
\phi_{1,3}'\otimes\bigwedge\fh&=\phi_{1,3}''\otimes\bigwedge\fh=\phi_{1,3}'+\phi_{1,3}''+\phi_{2,2};\\
\phi_{2,2}\otimes\bigwedge\fh&=2\phi_{2,2}+\phi_{2,1}+\phi_{1,3}'+\phi_{1,3}''.
\end{aligned}
\end{equation}
It follows that $\C F_\cusp(G_2)\subseteq\Theta^{-1}(\mathbf 0).$ Notice that in fact $\phi_{2,2}$ is sufficient for this conclusion.

\medskip

Type $\mathbf{F_4}$. The relevant $\sigma$ are $1_2$, $1_3$, $4_1$, $6_1$, $4_3$, and $4_4$. The wedge representations are (in order): $1_1$, $4_2$, $6_2$, $4_5$, and $1_4$. We find
\begin{align}
1_2\otimes\bigwedge\fh&=1_3\otimes\bigwedge\fh=1_2+4_3+6_1+4_4+1_3;\\
4_1\otimes\bigwedge\fh&=2\cdot 4_1+6_2+2\cdot 16_1+12_1+6_1;\\
6_1\otimes\bigwedge\fh&=3\cdot 6_1+2\cdot 4_3+2\cdot 4_4+2\cdot 16_1+1_2+1_3+4_1+9_2+9_3;\\
4_3\otimes\bigwedge\fh&=4_4\otimes\bigwedge\fh=2\cdot 4_3+2\cdot 4_4+1_2+9_2+2\cdot 6_1+16_1+1_3+9_3.
\end{align}
Taking the union of all the irreducible representations that occur in these decompositions, it follows that $\C F_\cusp(F_4)\subseteq\Theta^{-1}(\mathbf 0).$

\medskip

Type $\mathbf{E_6}$. The relevant $\sigma$ is $10_s$. The wedge representations are (in order): $\phi_{1,0}$, $\phi_{6,1}$, $\phi_{15,5}$, $\phi_{20,10}$, $\phi_{15,17}$, $\phi_{6,25}$, and $\phi_{1,36}$, in Carter's notation \cite{Ca}. We find
\begin{equation}
10_s\otimes\bigwedge\fh=3\cdot 10_s+4\cdot 60_s+3\cdot 90_s+20_s+80_s.
\end{equation}
It follows that $\C F_\cusp(E_6)\subseteq\Theta^{-1}(\mathbf 0).$

\medskip

Type $\mathbf{E_8}$. The relevant $\sigma$ are $168_y$ and $420_y$. The wedge representations are (in order): $\phi_{1,0}$, $\phi_{8,1}$, $\phi_{28,8}$, $\phi_{56,19}$, $\phi_{70,32}$, $\phi_{56,49}$, $\phi_{28,68}$, $\phi_{8,91}$, and $\phi_{1,120}$, in Carter's notation \cite{Ca}. We find
\begin{equation}
\begin{aligned}
168_y\otimes\bigwedge\fh&=3\cdot 168_y+4\cdot 1344_w+3\cdot 420_y+3\cdot 1134_y+2\cdot 3150_y+2\cdot 448_w\\
&+2\cdot 2016_w+2\cdot 5600_w+70_y+1400_y+1680_y+2688_y+4200_y;\\
420_y\otimes\bigwedge\fh&=5\cdot 420_y+6\cdot 1344_w+4\cdot 2016_w+3\cdot 168_y+4\cdot 1134_y+3\cdot 2688_y\\
&+4\cdot 3150_y+3\cdot 4200_y+2\cdot 448_w+4\cdot 5600_w+2\cdot 7168_w+70_y\\
&+1400_y+1680_y+4536_y+5670_y.
\end{aligned}
\end{equation}
It follows that $\C F_{\cusp}(W)\setminus \{4480_y\}\subseteq \Theta^{-1}(\mathbf 0)$.

\bigskip

For the opposite inclusions, suppose that a $W$-type $\tau$ is contained in $\Theta^{-1}(\mathbf 0)=(\zeta^*_{0,1})^{-1}(\mathbf 0)$.  In particular, by Proposition \ref{p:square-dirac}, we have
\begin{equation}
\sigma(\Omega_{W,1})=-\langle\Omega_{\bfH_{0,1}},\mathbf 0\rangle=0,
\end{equation}
where $\langle\Omega_{\bfH_{0,1}},\mathbf 0\rangle$ denote the natural evaluation pairing between  $Z(\bfH_{0,1})$ and \newline $\Spec Z(\bfH_{0,1})=X_1(W).$ But then Theorem \ref{t:scalar}, implies that
\begin{equation}
a(\C F)=a(\C F\otimes\sgn),
\end{equation}
where $\C F$ is the family containing $\sigma$. For exceptional Weyl groups $G_2$, $F_4$, and $E_6$, one can immediately see that the only families that have this property are the cuspidal ones.

In $E_8$, $\C F_\cusp$ is not the only family with the property that $\C F=\C F\otimes \sgn$. The other family with this property is a singleton, $\{2100_y\}$. Hence, this argument implies that $\Theta^{-1}(\mathbf 0)\subseteq \C F_\cusp(E_8)\cup \{2100_y\}.$

\end{proof}

\begin{remark}
The method of one-$W$-type modules doesn't yield any results for type $E_7$. The last part of the proof of Theorem \ref{t:CM-cusp} applies however to $E_7$ too and it gives
\begin{equation}
\{512_a'\}\subseteq\Theta^{-1}(\Theta(512_a'))\subseteq\{512_a',512_a\}=\C F_\cusp(E_7).
\end{equation}
\end{remark}

\begin{remark}\label{r:F4}
When the root system is not simply-laced ($B_n$, $F_4$, $G_2$), Theorem \ref{t:class} together with the same method as above can be used to construct nontrivial (cuspidal) Calogero-Moser cells. For example, in the case of $F_4$, the fact that the $W$-type $4_1$ extends to a one-$W$-type $\bfH_{0,c}$-module for all parameters $c$ implies that {\it for all parameters} $c=(c_\ell,c_s)$, the cuspidal CM cell $\Theta^{-1}(\mathbf 0)$ contains the set 
\begin{equation}\label{e:F4-unequal}
\{4_1, 6_1, 6_2,16_1,12_1\}.
\end{equation}
In fact, for all but finitely many lines in this two-dimensional parameter space (see Theorem \ref{t:class} for the list of non-generic values of $(c_\ell,c_s)$), $\Theta^{-1}(\mathbf 0)$ equals the set in (\ref{e:F4-unequal}).
\end{remark}

\ifx\undefined\bysame
\newcommand{\bysame}{\leavevmode\hbox to3em{\hrulefill}\,}
\fi

\end{document}